\documentclass[10pt]{amsart}
\usepackage{amssymb}
\usepackage{bm}
\usepackage{amsfonts}
\usepackage{amssymb,amsfonts,amsmath,amsthm,graphicx}
\usepackage{comment}
\usepackage{slashbox,mathrsfs}
\usepackage[all,poly]{xy}
\usepackage{float}
\usepackage{extarrows}

\allowdisplaybreaks

\begin{document}
\theoremstyle{plain}
\newtheorem{thm}{Theorem}[section]
\newtheorem{theorem}[thm]{Theorem}
\newtheorem*{theorem2}{Theorem}
\newtheorem{lemma}[thm]{Lemma}
\newtheorem{corollary}[thm]{Corollary}
\newtheorem{corollary*}[thm]{Corollary*}
\newtheorem{proposition}[thm]{Proposition}
\newtheorem{proposition*}[thm]{Proposition*}
\newtheorem{conjecture}[thm]{Conjecture}
\theoremstyle{definition}
\newtheorem{construction}[thm]{Construction}
\newtheorem{notations}[thm]{Notations}
\newtheorem{question}[thm]{Question}
\newtheorem{problem}[thm]{Problem}
\newtheorem{remark}[thm]{Remark}
\newtheorem{remarks}[thm]{Remarks}
\newtheorem{definition}[thm]{Definition}
\newtheorem{claim}[thm]{Claim}
\newtheorem{assumption}[thm]{Assumption}
\newtheorem{assumptions}[thm]{Assumptions}
\newtheorem{properties}[thm]{Properties}
\newtheorem{example}[thm]{Example}
\newtheorem{comments}[thm]{Comments}
\newtheorem{blank}[thm]{}
\newtheorem{observation}[thm]{Observation}
\newtheorem{defn-thm}[thm]{Definition-Theorem}

\def\supp{\operatorname{supp}}


\title[On a formula of Gammelgaard for Berezin-Toeplitz quantization]{On a formula of Gammelgaard for Berezin-Toeplitz quantization}

\author{Hao Xu}
        \address{Center of Mathematical Sciences, Zhejiang University, Hangzhou, Zhejiang 310027, China;
        Department of Mathematics, Harvard University, Cambridge, MA 02138, USA}
        \email{haoxu@math.harvard.edu}

        \begin{abstract}
        We give a proof of a slightly refined version of Gammelgaard's graph theoretic formula for Berezin-Toeplitz quantization on
        (pseudo-)K\"ahler manifolds.
        Our proof has the merit of giving an alternative approach to
        Karabegov-Schlichenmaier's identification theorem. We also identify the dual Karabegov-Bordemann-Waldmann star product.
        \end{abstract}


    \maketitle

\section{Introduction}

Deformation quantization on a symplectic manifold $M$ was introduced by Bayen et al. \cite{BFF} as a deformation of the usual
pointwise product of $C^{\infty}(M)$ into a noncommutative associative
$\star$-product of the formal series $C^\infty(M)[[\nu]]$. The celebrated work of
Kontsevich \cite{Kon} completely solved existence and classification of star-products
on general Poisson manifolds. See \cite{DS} for a comprehensive survey of deformation quantization on
symplectic and Poisson manifolds.

This paper will restrict to study differentiable deformation quantization with separation of variables on a K\"ahler manifold
(see below for definitions).
Let $(M,g)$ be a K\"ahler manifold of dimension $n$. On a coordinate
chart $\Omega$, the K\"ahler form is given by
$$\omega_g=\sqrt{-1}\sum_{i,j=1}^n
g_{i\overline{j}}dz_i\wedge dz_{\overline{j}}.$$ If $\Omega$ is
contractible, there exists a K\"{a}hler potential $\Phi$ satisfying
$$\partial\overline\partial\Phi=\sum_{i,j=1}^n
g_{i\overline{j}}dz_i\wedge dz_{\overline{j}}.$$

Around any point $x\in M$, there exists a normal
coordinate system such that
\begin{equation} \label{eqnormal}
g_{i\bar j}(x)=\delta_{ij}, \qquad g_{i\bar j k_1\dots
k_r}(x)=g_{i\bar j \bar l_1\dots\bar l_r}(x)=0
\end{equation}
for all $r\leq N\in \mathbb N$, where $N$ can be chosen arbitrary
large and $g_{i\bar j k_1\dots k_r}=\partial_{k_1\dots k_r}g_{i\bar
j}$.

The canonical Poisson bracket of two functions $f_1,f_2\in C^\infty(M)$ is given by
\begin{equation}
\{f_1,f_2\}=i g^{k\bar l}\left(\frac{\partial f_1}{\partial
z^k}\frac{\partial f_2}{\partial\bar z^l}-\frac{\partial
f_2}{\partial z^k}\frac{\partial f_1}{\partial\bar z^l}\right).
\end{equation}

Let $C^\infty(M)[[\nu]]$ denote the algebra of formal power series in
$h$ over $C^\infty(M)$. A star product is an associative $\mathbb
\mathbb C[[\nu]]$-bilinear product $\star$ such that $\forall
f_1,f_2\in C^\infty(M)$,
\begin{equation} \label{eqberdef}
f_1\star f_2=\sum_{j=0}^\infty \nu^j C_j(f_1,f_2),
\end{equation}
where the $\mathbb C$-bilinear operators $C_j$ satisfy
\begin{equation}
C_0(f_1,f_2)=f_1 f_2,\qquad C_1(f_1,f_2)-C_1(f_2,f_1)=i\{f_1,f_2\}.
\end{equation}

A star product is called \emph{differentiable}, if each $C_j$ is a bidifferential operator.
According to Karabegov \cite{Kar}, a star product has the property of \emph{separation of variables} (\emph{Wick type}), if it
satisfies $f\star h=f\cdot h$ and $h\star g=h\cdot g$ for any
locally defined antiholomorphic function $f$, holomorpihc function
$g$ and an arbitrary function $h$. If the role of holomorphic and antiholomorphic variables are swapped, we call it a star
product of \emph{anti-Wick type}. In particular, the Berezin-Toeplitz star product is of Wick type and the Berezin star product is of anti-Wick type.
They are dual opposite to each other.

Motivated by the work of Reshetikhin and Takhtajan \cite{RT}, Gammelgaard \cite{Gam} obtained a remarkable universal formula for any
star product with separation of variables corresponding to a given
classifying Karabegov form.
Using
Karabegov-Schlichenmaier's identification theorem \cite{KS} for Berezin-Toeplitz quantization (cf. Theorem \ref{thmKS}), Gammelgaard's
formula specializes to a graph
expansion for Berezin-Toeplitz quantizaion over acyclic graphs, which can be equivalently formulated in the following theorem. See Remark \ref{rm2}.
\begin{theorem} \label{main} The Berezin-Toeplitz star product $\star_{BT}$ on a K\"ahler manifold is given by
\begin{equation} \label{eqBT}
f_1\star_{BT} f_2(x)=\sum_{\Gamma=(V\cup\{f\},E)\in\mathcal G_{BT}^{ss}}\nu^{|E|-|V|}\frac{(-1)^{|E|}}{|{\rm Aut}(\Gamma)|}
\Gamma(f_1,f_2),
\end{equation}
where $\mathcal G_{BT}^{ss}$ (see \eqref{graphBTss} for the definition) is a certain subset of strongly connected semistable one-pointed
graphs and $\Gamma(f_1,f_2)$ is the partition
function of $\Gamma$ (see Definition \ref{partition2}).
\end{theorem}

In \cite{Xu2}, an explicit formula of Berezin transform in terms of strongly connected graphs was obtained, building on the works
of Engli\v s \cite{Eng}, Loi \cite{Loi}, and related results of Charles \cite{Cha}.
At a first look, acyclic graphs and strongly connected graphs are quite different.
As noted by Schlichenmaier \cite{Sch4},
it shall be interesting to clarify their relations.

We will give a purely graph theoretic proof of Theorem \ref{main} in \S \ref{BT} by developing a technique of computing the inverse to the Berezin transform.
Our proof also provides some clarifications to Gammelgaard's formula and
        Karabegov-Schlichenmaier's identification theorem for Berezin
        and Berezin-Toeplitz star products. In \S \ref{KBW}, we identify the Karabegov form of the dual Karabegov-Bordemann-Waldmann star product.

All graphs in this paper represent partial derivatives
of K\"ahler metrics and functions. The following two identities can be used to convert covariant derivatives of curvature tensors to partial derivatives
of metrics and vice versa.
\begin{gather}\label{eqcur1}
R_{i\bar jk\bar l} =-g_{i\bar j k\bar l}+g^{m\bar p}g_{m\bar j\bar
l}g_{ik\bar p},\\
\label{eqcur2}
T_{\alpha_1\dots\alpha_p;\gamma}=\partial_{\gamma}T_{\alpha_1\dots\alpha_p}-\sum_{i=1}^p
\Gamma_{\gamma\alpha_i}^{\delta}T_{\alpha_1\dots\alpha_{i-1}\delta\alpha_{i+1}\dots\alpha_p},
\end{gather}
where $T_{\alpha_1\dots\alpha_p}$ is a covariant tensor and $\Gamma_{\beta\gamma}^\alpha=0$ except
for $
\Gamma_{jk}^i=g^{i\bar l}g_{j\bar l k},\, \Gamma_{\bar j\bar
k}^{\bar i}=g^{l\bar i}g_{l\bar j\bar k}$.

We want to point out that for graph expressions in this paper,
we can either sum over stable, semistable or all strongly connected graphs.
Their difference is briefly indicated here. For a summation over stable graphs, the equation holds only at the center of
the normal coordinate system, but it is enough to uniquely recover the curvature tensor expression.
For a summation over semistable graphs, the equation holds globally, which is needed when
taking derivatives. Finally, we may enlarge the summation over all strongly connected graphs in order to simplify
the computation of derivatives (see Remark \ref{rm3}).

\

\noindent{\bf Acknowledgements}
The author thanks Professor Kefeng Liu, Hongwei Xu and Shing-Tung Yau
for their kind help during his years of graduate studies at Center of Mathematical Sciences (CMS),
Zhejiang University.
This paper is dedicated to CMS on the occasion of its 10th birthday.

\vskip 30pt
\section{Berezin transform} \label{berezin}

In this section, we briefly recall the asymptotic expansion of Berezin transform and related constructions
of star products. More detailed expositions and historical remarks can be found in recent nice surveys \cite{Sch3, Sch4}.

Let $\Phi(x,y)$ be an almost analytic extension of $\Phi(x)$ to a
neighborhood of the diagonal, i.e. $\bar\partial_x\Phi$ and
$\partial_y\Phi$ vanish to infinite order for $x=y$.
We can assume $\overline{\Phi(x,y)}=\Phi(y,x)$.

For $m > 0$, consider the weighted Bergman space of all
holomorphic function on $\Omega$ square-integrable with respect to
the measure $e^{-m \Phi}\frac{w_g^n}{n!}$.
We denote by $K_m(x,y)$ the reproducing kernel. Locally
\begin{equation}\label{eqb1}
K_m (x,y)\sim e^{m \Phi (x,y)} \sum^\infty_{k=0} B_k (x,y)
m^{n-k},
\end{equation}
which converges if $\Omega$ is a strongly pseudoconvex domain
with real analytic boundary (cf. \cite{Ber, Eng}). For
discussions on the convergence in the compact K\"ahler case, see
\cite{KS,Zel}.

The asymptotic expansion of the Bergman kernel in the setting when $\Omega$ is a compact
K\"ahler manifold was also extensively studied. For recent works, see e.g. \cite{DLM, DK, HM,
LL}.

The \emph{Berezin transform} is the integral operator
\begin{equation} \label{eqber}
I_m f(x)=\int_\Omega
f(y)\frac{|K_m(x,y)|^2}{K_m(x,x)}e^{-m
\Phi(y)}\frac{w_g^n(y)}{n!}.
\end{equation}
At any point for which $K_m (x,x)$ invertible, the integral
converges for each bounded measurable function $f$ on $\Omega$. Note
that \eqref{eqb1} implies that for any $x$, $K_m (x,x)\neq 0$
if $m$ is large enough.

The Berezin transform has an asymptotic expansion for
$m\rightarrow\infty$ (cf. \cite{Eng, KS}),
\begin{equation} \label{eqber3}
I^{(m)} f(x)=\sum^\infty_{k=0} Q_k f(x)m^{-k},
\end{equation}
where $Q_k$ are linear differential operators.

The Berezin star product was introduced by Berezin \cite{Ber}
through symbol calculus for linear operators on weighted Bergman
spaces (cf. \cite{CGR, Eng3, Sch3}). Karabegov \cite{Kar1} noted that
any star product with separation of variables can be constructed from a unique formal
Berezin transform. In our case, if we write
\begin{equation}
Q_j f=\sum_{\alpha,\beta\text{
multiindices}}c_{j\alpha\beta}\partial^{\alpha}\bar\partial^{\beta}f,
\end{equation}
then the coefficients of Berezin star product are given by
bilinear differential operators
\begin{equation} \label{eqkar}
C_j(f_1,f_2):=\sum_{\alpha,\beta}
c_{j\alpha\beta}(\bar\partial^{\beta}f_1)(\partial^{\alpha} f_2).
\end{equation}

The Berezin star product $\star_B$ is equivalent to the Berezin-Toeplitz star
product $\star_{BT}$ via the Berezin transform (cf. \cite{KS})
\begin{equation} \label{eqbt}
f_1\star_{BT} f_2=I^{-1}(I f_1\star_B I f_2),
\end{equation}
where $I:=I^{(1/\nu)}$ is obtained from substituting $m$ by $1/\nu$
in $I^{(m)}$.

Recall that the Toeplitz operator $T^{(m)}_f$ for $f\in C^\infty(M)$
is defined by
\begin{equation}
T^{(m)}_f:=\Pi^{(m)}(f\cdot):\quad H^0(M,L^m)\rightarrow H^0(M,L^m),
\end{equation}
where $\Pi^{(m)}:L^2(M,L^m)\rightarrow H^0(M,L^m)$ is the orthogonal
projection. See \cite{BMS} for a detailed study of their semiclassical properties. The following
celebrated theorem of Schlichenmaier \cite{Sch} shows that
Berezin-Toeplitz operator quantization and deformation quantization are closed related.
\begin{theorem}[\cite{Sch}]\label{sch}
The Berezin-Toeplitz
star product \eqref{eqbt} is the unique star product
\begin{equation}
f_1\star_{BT} f_2:=\sum_{j=0}^\infty \nu^j C^{BT}_j(f_1,f_2),
\end{equation}
such that the following asymptotic expansion holds
\begin{equation} \label{eqbt2}
T_{f_1}^{(m)}T_{f_2}^{(m)} \sim \sum_{j=0}^\infty m^{-j}
T_{C^{BT}_j(f_1,f_2)}^{(m)},\qquad m\rightarrow\infty.
\end{equation}
\end{theorem}

The Berezin transform was introduced by Berezin \cite{Ber2} for
symmetric domains in $\mathbb C^n$ and later extended by many authors
(see e.g. \cite{Eng0,Eng3}).
Karabegov and Schlichenmaier \cite{KS} proved the asymptotic
expansion of the Berezin transform for compact K\"ahler manifolds.

Berezin-Toeplitz quantization for compact K\"ahler manifolds was extensively studied in the
literature and had found many applications.
Karabegov \cite{Kar3} constructed an algebra of Toeplitz elements that is isomorphic to the algebra of Berezin-Toeplitz quantization.
Ma and Marinescu
\cite{MM2} developed the theory of Toeplitz operators on
symplectic manifolds twisted with a vector bundle.
Zelditch \cite{Zel2} studied quantizations of symplectic maps on compact K\"ahler manifolds
and uncovered a connection between Berezin-Toeplitz quantization and quantum chaos.
Andersen \cite{And} proved asymptotic faithfulness of the
mapping class groups action on Verlinde bundles by using Berezin-Toeplitz technique (cf. also \cite{Sch2}).

\vskip 30pt
\section{Differential operators encoded by graphs} \label{graph}

Throughout this paper, a {\it digraph}, or simply a graph, $G=(V,E)$ is defined to be a directed multi-graph, i.e. it has finite number of vertices and edges with multi-edges and
loops allowed.

Recall the definition of stable and semistable graphs in \cite{Xu, Xu2}. These graphs were used to represent Weyl
invariants, which encode the coefficients of the asymptotic
expansion.
\begin{definition}
We call a vertex $v$ of a digraph $G$ \emph{semistable} if we have
$$\deg^-(v)\geq1,\ \deg^+(v)\geq1,\ \deg^-(v)+\deg^+(v)\geq3.$$
$v$ is call \emph{stable} if $\deg^-(v)\geq2,\ \deg^+(v)\geq2$.
 \end{definition}

\begin{definition}\label{dfdot}
An \emph{$m$-pointed graph} $\Gamma=(V\cup\{f_1,\dots,f_m\},E)$ is
defined to be a digraph with $m$ distinguished vertex labeled by $f_1,\dots,f_m$. An automorphism of $\Gamma$ fixes
each of the $m$ distinguished vertices and its automorphism group is denoted by ${\rm Aut}(\Gamma)$. $\Gamma$ is called semistable (stable)
if each ordinary vertex $v\in V$ is semistable (stable).
We denote by
\begin{verse}
$V(\Gamma)=V\cup\{f_1,\dots,f_m\}$;

$\Gamma_-$ the subgraph
of $\Gamma$ obtained by removing distinguished vertices;

$w(\Gamma)=|E|-|V|$ the weight of $\Gamma$;

$\overline{\mathcal G}_m$ the set of all $m$-pointed graphs;

$\mathcal G_m^{scon}$ the set of all strongly connected $m$-pointed graphs;

$\mathcal G_m^{ss}$ the set of all strongly connected semistable $m$-pointed graphs;

$\mathcal G_m$ the set of all strongly connected stable $m$-pointed graphs.
\end{verse}
For an $m$-pointed graph $G$, we denote by $\dot G$ the one-pointed graph obtained by
merging the $m$ distinguished points of $G$ into one point. $G$ is called strongly connected if $\dot G$ is strongly connected.
It is obvious that $\mathcal G_m\subset\mathcal G_m^{ss}\subset\mathcal G_m^{scon}\subset\overline{\mathcal G}_m$.

We also denote by
\begin{verse}
$\overline{\mathcal G}_m(k)$ the set of all $m$-pointed graphs with weight $k$;

$\overline{\mathcal G}=\cup_{m\geq0}\overline{\mathcal G}_m$ the set of all pointed graphs.

\qquad $\vdots$
\end{verse}
The same notation applies to other sets of graphs listed above.
\end{definition}

The one-pointed graph in Figure \ref{figweyl} represents an Weyl invariant $g_{i \bar i k \bar l p} g_{j \bar j l \bar k
\bar q} f_{q\bar p}$ in partial derivatives of metrics and functions. Note that $(i,\bar i),\,(j, \bar j)$ etc. are paired indices to be
contracted. The weight of a
directed edge is the number of multi-edges. The number attached to a
vertex denotes the number of its self-loops. A vertex without loops
will be denoted by a small hollow circle $\circ$. The distinguished
vertex is denoted by a solid circle $\bullet$.
\begin{figure}[h]
$\xymatrix@C=4mm@R=7mm{
                & \bullet  \ar[dr]^{1}             \\
         *+[o][F-]{1} \ar[ur]^{1} \ar@/^0.4pc/[rr]^{1} & &  *+[o][F-]{1}  \ar@/^0.4pc/[ll]^{1}
         }
$ \caption{A strongly connected one-pointed graph} \label{figweyl}
\end{figure}
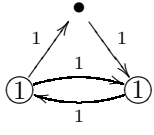

The following theorem on graph theoretic formulae for Bergman kernel and Berezin transform will be used in the sequel.
\begin{equation}
B(x)=\sum_{k\geq0}\nu^k B_k(x),\qquad I f(x)=\sum_{k\geq0}\nu^k Q_k f(x).
\end{equation}
\begin{theorem}[\cite{Xu}] On a K\"ahler manifold $M$, the Bergman kernel has the expansion
\begin{equation}\label{eqb2}
B(x)=\exp\left(\sum_{G=(V,E)\in\mathcal G_0^{ss}}\nu^{|E|-|V|}\; \frac{-\det(A(G)-I)}{|{\rm Aut}(G)|}\,G\right),
\end{equation}
where $G$ runs over all strongly connected semistable graphs.
\end{theorem}
In particular, there are two strongly connected semistable graphs of weight $1$, so we have
\begin{equation}
B(x)=\exp\left(1+\nu\left(-\frac{1}{2}\big[\xymatrix{*+[o][F-]{2}}\big]+\frac12\bigg[\xymatrix{ \circ  \ar@/^/[r]^2 & \circ \ar@/^/[l]^1
         }\bigg]\right)+O(\nu^2)\right).
\end{equation}

\begin{theorem}[\cite{Xu2}] On a K\"ahler manifold $M$, the Berezin transform has the expansion
\begin{equation}\label{eqb3}
I (f)=\sum_{\Gamma=(V\cup\{f\},E)\in\mathcal G_1^{scon}}\nu^{|E|-|V|}\;\frac{\det(A(\Gamma_-)-I)}{|{\rm Aut}(\Gamma)|}\,\Gamma.
\end{equation}
\end{theorem}
In particular, coefficients up to $\nu^3$ in \eqref{eqb3} (keeping only stable graphs) were obtained in \cite{Xu2}.
\begin{multline}\label{eqb9}
I(f)=f+\nu\Big[\xymatrix{\bullet
              \ar@(ur,dr)[]^{1}}\Big]
+\frac{\nu^2}2\Big[\xymatrix{\bullet
           \ar@(ur,dr)[]^{2}}\Big]
+\nu^3\Bigg(\frac{1}{6}\Big[\xymatrix{\bullet \ar@(ur,dr)[]^{3}}\Big]-
       \frac{1}{4}\bigg[\xymatrix{ \circ \ar@/^/[r]^2 & \bullet \ar@/^/[l]^2
         }\bigg]
\\-\frac{1}{2}\Bigg[\begin{minipage}{0.6in}$\xymatrix@C=2mm@R=5mm{
                & \bullet \ar[dr]^{1}             \\
         \circ \ar[ur]^{1} \ar@/^0.3pc/[rr]^{1} & &  \circ  \ar@/^0.3pc/[ll]^{2}
         }$
         \end{minipage}\Bigg]-\Bigg[\begin{minipage}{0.65in}
             $\xymatrix@C=2mm@R=5mm{
                & \bullet \ar@/^-0.3pc/@<0.2ex>[dl]^{1}             \\
         \circ \ar@/^0.3pc/@<0.7ex>[ur]^{1} \ar@/^-0.3pc/@<0.2ex>[rr]^{1} & & *+[o][F-]{1}
         \ar@/^0.3pc/@<0.7ex>[ll]^{1}
         }$
         \end{minipage}\Bigg]+\frac12\bigg[\xymatrix{
        *+[o][F-]{2} \ar@/^/[r]^1
         &
        \bullet \ar@/^/[l]^1} \bigg]\Bigg)+O(\nu^4).
\end{multline}

Note that in \eqref{eqb3} we sum over all strongly connected one-point graph, no matter the graph is semistable or not. The reason is that
it will simplify the computation of the inverse Berezin transform as explained in Remark \ref{rm3}.

Fix a normal coordinate system around $x$ on a K\"ahler manifold $M$, then \eqref{eqb2} and \eqref{eqb3} still hold at $x$
if we only sum over those stable graphs.

We now define three special subsets of $\mathcal G_1$:
\begin{align}
\mathcal G_B=\{\Gamma\in\mathcal G_1\mid & \text{ 1 is not an
eigenvalue of } A(\Gamma_-)\},\\ \label{graphBT}
\mathcal G_{BT}=\{\Gamma\in\mathcal G_1\mid & \text{ each SCC of } \Gamma_- \text{ is either a single vertex }
\\& \text{ or a linear digraph}\}, \nonumber
\\ \label{graphS}
\mathcal G_S=\{\Gamma\in\mathcal G_1\mid & \text{ each SCC of } \Gamma_- \text{ is a single vertex without loops} \},
\end{align}
where SCC is an abbreviation for \emph{strongly connected component}. More explicitly, $\Gamma\in\mathcal G_{BT}$
if each SCC of $\Gamma_-$ belongs to
\begin{equation}\label{lineargraph}
\circ\quad \xymatrix{*+[o][F-]{1}}\quad \xymatrix{
        \circ \ar@/^/[r]^1
         &
        \circ \ar@/^/[l]^1}\quad
        \begin{minipage}{0.6in}$\xymatrix@C=3mm@R=6mm{
                & \circ \ar[dr]^{1}             \\
        \circ  \ar[ur]^{1}  & &  \circ \ar[ll]_{1}
         }$\end{minipage}\quad \cdots
\end{equation}
And $\Gamma\in\mathcal G_S$ can also be characterized by saying that $\Gamma_-$ is a directed acyclic graph (DAG).

\begin{remark}
For later use, we also need to relax the condition ``stable'' to semistable or just strongly connected in the above definition.
The corresponding sets are denoted by $\mathcal G_B^{ss},\mathcal G_{BT}^{ss},\mathcal G_S^{ss}$ and $\mathcal G_B^{scon},\mathcal G_{BT}^{scon},\mathcal G_S^{scon}$. For example,
\begin{align}
\mathcal G_{BT}^{scon}=\{\Gamma\in\mathcal G_1^{scon}\mid & \text{ each SCC of } \Gamma_- \text{ is either a single vertex }
\\& \text{ or a linear digraph}\}, \nonumber
\\
\mathcal G_{BT}^{ss}=\{\Gamma\in\mathcal G_{BT}^{scon}\mid & \ \Gamma \text{ is semistable}\}.\label{graphBTss}
\end{align}
\end{remark}

For any $k\geq0$, the sets $\mathcal G_B(k),\mathcal G_{BT}(k),\mathcal G_S(k)$ are respectively in one-to-one correspondence with
weight $k$ terms of Berezin, Berezin-Toeplitz and Karabegov-Bordemann-Waldmann star products.

We have computed the cardinalities of these sets when $k\leq6$ in
Table \ref{tb1}.

\begin{table}[h]
\caption{Numbers of strongly connected stable one-pointed graphs} \label{tb1}
\begin{tabular}{|c||c|c|c|c|c|c|c|}
\hline
    $k$                                            &$0$&$1$&$2$&$3$ &$4$  &$5$   &$6$
\\\hline       $|\mathcal G_1(k)|$      &$1$&$1$&$2$&$9$ &$61$ &$538$ &$5906$
\\\hline       $|\mathcal G_{B}(k)|$                  &$1$&$1$&$1$&$5$ &$36$ &$331$ &$3704$
\\\hline       $|\mathcal G_{BT}(k)|$                  &$1$&$1$&$2$&$6$ &$24$ &$112$ &$620$
\\\hline       $|\mathcal G_{S}(k)|$                  &$1$&$1$&$1$&$2$ &$5$ &$15$ &$54$
\\\hline
\end{tabular}
\end{table}

\begin{definition} \label{partition}
Let $\Gamma\in\overline{\mathcal G}_1$ be a one-pointed graph, the {\it partition function}
$D_{\Gamma}(f_1,f_2)$ is defined to be a Weyl invariant generated
from $\Gamma$ by replacing the vertex $f$ in with two vertices $f_1$
and $f_2$, such that all outward edges of $f$ are connected to $f_1$
and all inward edges of $f$ are connected to $f_2$.

Let $H\in\overline{\mathcal G}$ be an arbitrary pointed graph, we
define $D_{\Gamma}(H)$ to be the Weyl invariant generated by
replacing $f$ in $\Gamma$ with $H$.
\end{definition}

For example, if $\Gamma$ is the graph in Figure \ref{figweyl}, then
\begin{gather}
D_{\Gamma}(f_1,f_2)=g_{i \bar i k \bar l p} g_{j \bar j l \bar k
\bar q} \partial_{q} f_1 \partial_{\bar p} f_2,\\
D_{\Gamma}(H)=g_{i \bar i k \bar l p} g_{j \bar j l \bar k \bar q}
\partial_{q\bar p} H.\label{eqb4}
\end{gather}

Similarly for two pointed graphs $H_1,H_2\in\overline{\mathcal G}$, we
define $D_{\Gamma}(H_1,H_2)$ to be the Weyl invariant generated
by replacing $f_1,f_2$ with $H_1,H_2$ in $D_{\Gamma}(f_1,f_2)$. We may linearly extend $D_\Gamma(\cdot)$ to be defined on linear combination of graphs.
We also use the notation $D_{\Gamma}^{op}(H_1,H_2)=D_{\Gamma}(H_2,H_1)$.

When expanding $\partial_{q\bar p} H$ in \eqref{eqb4}, we may need to take derivatives of $g^{i\bar j}$,
\begin{equation}\label{eqginv}
\partial_{\alpha}g^{i\bar j}=-g^{p\bar j}g^{i\bar q}g_{p\bar
q \alpha},
\end{equation}
where new vertex appears as shown in Figure \ref{figedge}.
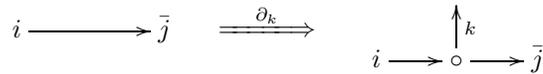
\begin{figure}[h]
\begin{tabular}{c}$\xymatrix@C=7mm@R=5mm{  i \ar[rr] & & \bar j}$ \end{tabular}\quad
 $\xLongrightarrow{\quad \partial_k\quad }$ \quad
 \begin{tabular}{c}
 $\xymatrix@C=7mm@R=5mm{  & &\\  i \ar[r] &\circ \ar[r] \ar[u]_k & \bar j}$ \end{tabular}
\caption{Illustration of $\partial_{k}g^{i\bar j}=-g^{p\bar j}g^{i\bar q}g_{p\bar
q k}$}
 \label{figedge}
\end{figure}

\begin{remark}\label{rm3} In the recursive computation of the inverse Berezin tranform $I^{-1}$ from $I^{-1}I=id$ or $I\cdot I^{-1}=id$,
we need to handle terms like \eqref{eqb4}. The possible action on the edge (cf. Figure \ref{figedge}) made the computation
much more complicated. However, thanks to the special structure of $I$ and $I^{-1}$, we may extend their summations over all
strongly connected graphs (cf. \eqref{eqb3}, \eqref{eqb5}) such that when taking derivatives on a graph, the derivatives only go to
vertices. This follows from Lemma \ref{matrix} (note the minus sign in \eqref{eqginv}) and the fact that any strongly connected graph can be obtained by adding finite
number of vertices to edges of a semistable graph. Adding a vertex to an edge may change the automorphism group of a graph, but this will
be compensated when taking derivatives by Leibniz rule. This assertion can be made rigorous by following the argument of Lemma \ref{graph2}.
For example, the automorphism group of the left-hand side (LHS) graph in Figure \ref{figedge2}
has order $2$. It becomes rigid (i.e. have no nontrivial automorphism) after adding a vertex on an edge. But there are exactly two ways
of adding a vertex to the LHS graph to turn it into the RHS graph.
\begin{figure}[h]
\begin{tabular}{c}$\xymatrix@C=7mm@R=5mm{  \circ \ar@/^0.6pc/[rr] & & \circ \ar@/^0.6pc/[ll]}$ \end{tabular}
\qquad
\begin{tabular}{c}
 $\xymatrix@C=7mm@R=0mm{
 &\circ \ar@/^0.2pc/[dr]&
 \\
 \circ \ar@/^0.2pc/[ur] & & \circ \ar@/^/[ll]
 \\
 & & }$
\end{tabular}
\caption{Adding a vertex to an edge}
 \label{figedge2}
\end{figure}
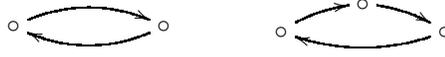
In fact, throughout the paper, when dealing with $D_\Gamma(H)$ and $D_\Gamma(H_1,H_2)$,
it is sufficient to stipulate that no derivatives will be taken on edges. Thus we will be able to use Lemma \ref{graph2}.
\end{remark}

For convenience, we introduce the following notation.
\begin{definition}\label{partition2} Let $\Gamma\in\overline{\mathcal G}_1$ and $H,H_1,H_2\in\overline{\mathcal G}$.
Define $\Gamma(H)$ and $\Gamma(H_1,H_2)$ respectively to be the summation of those graph terms in the expansion of $D_\Gamma(H)$ and $D_\Gamma(H_1,H_2)$
with
derivatives only acting on the vertices of $H, H_1, H_2$. Namely we discard all graphs in $D_\Gamma(H)$ and $D_\Gamma(H_1,H_2)$ if some vertex of it
was created through taking derivative on an edge. We also use the notation $\Gamma^{op}(H_1,H_2)=\Gamma(H_2,H_1)$.
\end{definition}

\begin{lemma} \label{matrix}
Let $\Gamma\in\overline{\mathcal G}_1$ be a one-pointed graph and $\Gamma'$ be the graph obtained by adding a vertex
to an edge $e=(v_1,v_2)$ of $\Gamma$. Then $\det(A(\Gamma'_-)-I)=-\det(A(\Gamma_-)-I)$.
\end{lemma}
\begin{proof}
First we assume that both $v_1,v_2$ are not the distinguished vertex of $\Gamma$. Let $v_1,\dots,v_n$ be the vertices of $\Gamma_-$ and
$a_{ij}$ be the number of directed edges from $v_i$ to $v_j$. Then
$$A(\Gamma'_-)-I=\begin{bmatrix}\begin{smallmatrix}
 -1 & 0 & 1 & \dots & 0\\
 1 & a_{11}-1 & a_{12}-1 &  & \\
 0 & a_{21} & a_{22}-1 & &   \\
 \vdots   &   &  & \ddots  & \vdots\\
 0&   & \dots &  & a_{nn}-1
\end{smallmatrix}\end{bmatrix}=\begin{bmatrix}\begin{smallmatrix}
 -1 & 0 & 0 & \dots & 0\\
 1 & a_{11}-1 & a_{12} &  & \\
 0 & a_{21} & a_{22}-1 & &   \\
 \vdots   &   &  & \ddots  & \vdots\\
 0&   & \dots &  & a_{nn}-1
\end{smallmatrix}\end{bmatrix},$$
which implies that $\det(A(\Gamma'_-)-I)=-\det(A(\Gamma_-)-I)$.

If at least one of $v_1,v_2$ is the distinguished vertex, the lemma is obvious.
\end{proof}

Let $G,H\in\overline{\mathcal G}_m$ be $m$-pointed graphs and $H$ be a subgraph of $G$ requiring that the
$m$ distinguished points of $H$ are exactly the $m$ distinguished points of $G$. We define $G/H$ to be a one-pointed graph
obtained from $G$ by contracting $H$ to a point, which will be the distinguished vertex
of $G/H$. The ordinary vertices of $G/H$ are just the vertices not in $H$ and edges of $G/H$ are just edges
not in $H$.

Let $\Gamma\in\overline{\mathcal G}_1$ and $G,H\in\overline{\mathcal G}_m$. We define a natural number
\begin{equation}\label{alpha}
\alpha(\Gamma,H;G):=\#\{\text{subgraphs } H'\text{ of } G\mid H'\cong H, G/H'\cong\Gamma\}
\end{equation}

\begin{lemma} \label{graph2}
Let $\Gamma\in\overline{\mathcal G}_1$ and $H\in\overline{\mathcal G}$.
Then
\begin{equation}
\frac{1}{|{\rm Aut}(\Gamma)| |{\rm
Aut}(H)|}\Gamma(H)=\sum_{G}\frac{\alpha(\Gamma,H;G)}{|{\rm Aut}(G)|}G,
\end{equation}
where $G$ in the summation runs over isomorphism classes of graphs appearing in the
expansion of $\Gamma(H)$ by Leibniz rule.
\end{lemma}
\begin{proof}
Note that the group ${\rm Aut}(\Gamma)\times {\rm Aut}(H)$
has a natural action on the multiset
of all graphs in the expansion of $\Gamma(H)$ by Leibniz rule. Then it is not
difficult to see that the set of orbits corresponds to isomorphism
classes of graphs and the isotropy group at a graph $G$ is ${\rm Aut}(G)_{H'}$,
the subgroup of ${\rm Aut}(G)$ that leaves invariant a subgraph $H'$ in
the set at the right-hand side of \eqref{alpha}. So we have
$$\Gamma(H)=\sum_G \frac{|{\rm Aut}(\Gamma)| |{\rm
Aut}(H)|}{|{\rm Aut}(G)_{H'}|}G=\sum_G \frac{\alpha(\Gamma,H;G)|{\rm Aut}(\Gamma)| |{\rm
Aut}(H)|}{|{\rm Aut}(G)|}G,$$
as claimed.
\end{proof}

\begin{definition}\label{defgraph}
Given a digraph $G$, we call a subgraph $L$ of $G$ a \emph{generalized linear subgraph} if
each connected component
of $L$ is either a vertex without loops or a linear graph (i.e. belongs to the graphs in \eqref{lineargraph}).
We denote by $\mathscr L(G)$ the set of all spanning generalized linear
subgraphs $L$ of $G$.
\end{definition}
The following theorem is a corollary of the so-called coefficient theorem in spectral graph theory.
\begin{theorem} \label{ct}
For a digraph $G$ with $n$ vertices, we have
\begin{equation}
\det(A(G)-I)=\sum_{L\in\mathscr L(G)} (-1)^{n+p(L)}=\sum_{L\in\mathscr L(G)}\prod_{C\in L} (-1)^{\ell(C)+1},
\end{equation}
where
$p(L)$ denotes the number of components of $L$ and $C$ runs over components of $L$ and
$\ell(C)$ denotes the length of $C$. We regard a single vertex as a 0-cycle and a loop as a 1-cycle.
\end{theorem}

\vskip 30pt
\section{Berezin-Toeplitz quantization} \label{BT}

The first few terms of Berezin-Toeplitz quantization have been computed in \cite{Eng3,KS,MM,Xu2}. Since Berezin-Toeplitz quantization
corresponds to the inverse Berezin transform, the following theorem
completely determines the structure of Berezin-Toeplitz quantization and implies \eqref{eqBT} in Theorem \ref{main}.
\begin{theorem} On a K\"ahler manifold, the inverse Berezin transform is given by
\begin{equation}\label{eqb5}
I^{-1}(f)=\sum_{\Gamma=(V\cup\{f\},E)\in\mathcal G_{BT}^{scon}}\nu^{|E|-|V|}\frac{(-1)^{|E|}}{|{\rm Aut}(\Gamma)|}
\,\Gamma.
\end{equation}
\end{theorem}
In particular, coefficients up to $\nu^3$ in \eqref{eqb5} (keeping only stable graphs) were obtained in \cite{Xu2}.
\begin{multline}\label{eqb8}
I^{-1}(f)=f-\nu\Big[\xymatrix{\bullet
              \ar@(ur,dr)[]^{1}}\Big]
+\nu^2\left(\frac12\Big[\xymatrix{\bullet
           \ar@(ur,dr)[]^{2}}\Big]-\bigg[\xymatrix{
        *+[o][F-]{1} \ar@/^/[r]^1
         &
        \bullet \ar@/^/[l]^1} \bigg]\right)\\
+\nu^3\Bigg(-\frac{1}{6}\Big[\xymatrix{\bullet \ar@(ur,dr)[]^{3}}\Big]+\bigg[\xymatrix{
        *+[o][F-]{1} \ar@/^/[r]^1
         &
        \bullet\, 1 \ar@/^/[l]^1} \bigg]+\frac{1}{4}\bigg[\xymatrix{ \circ \ar@/^/[r]^2 & \bullet \ar@/^/[l]^2
         }\bigg]+\frac{1}{2}\bigg[\xymatrix{
        *+[o][F-]{1} \ar@/^/[r]^1
         &
        \bullet \ar@/^/[l]^2} \bigg]
\\+\frac{1}{2}\bigg[\xymatrix{
        *+[o][F-]{1} \ar@/^/[r]^2
         &
        \bullet \ar@/^/[l]^1} \bigg]-\Bigg[\begin{minipage}{0.6in}$\xymatrix@C=2mm@R=5mm{
                & \bullet \ar[dr]^{1}             \\
        *+[o][F-]{1} \ar[ur]^{1}  & &  *+[o][F-]{1} \ar[ll]^{1}
         }$
         \end{minipage}\Bigg]\Bigg)+O(\nu^4).
\end{multline}
The 24 strongly connected stable graphs in $\mathcal G_{BT}(4)$ are listed in
Table \ref{tb2}.
\begin{table}[h] \footnotesize
\centering \caption{The 24 graphs in $\mathcal G_{BT}(4)$}
\label{tb2}
\begin{tabular}{|c|c|c|c|c|c|}

\hline $\xymatrix{\bullet \,4}$
     & $\xymatrix{
        \circ \ar@/^/[r]^2
         &
        \bullet\,1 \ar@/^/[l]^2} $
     & $\xymatrix{
        \circ \ar@/^/[r]^2
         &
        \bullet \ar@/^/[l]^3} $
     & $\xymatrix{
        \circ \ar@/^/[r]^3
         &
        \bullet \ar@/^/[l]^2} $
     & $\xymatrix{
         *+[o][F-]{1} \ar@/^/[r]^1
         &
        \bullet\, 2 \ar@/^/[l]^1} $
     & $\xymatrix{
         *+[o][F-]{1} \ar@/^/[r]^1
         &
        \bullet\,1 \ar@/^/[l]^2} $
\\
\hline $1/24$ & $-1/4$ & $-1/12$ & $-1/12$ & $-1/2$ & $-1/2$
\\
\hline $\xymatrix{
         *+[o][F-]{1} \ar@/^/[r]^1
         &
        \bullet \ar@/^/[l]^3} $
     & $\xymatrix{
         *+[o][F-]{1} \ar@/^/[r]^2
         &
        \bullet\,1 \ar@/^/[l]^1} $
     & $\xymatrix{
         *+[o][F-]{1} \ar@/^/[r]^2
         &
        \bullet \ar@/^/[l]^2} $
     & $\xymatrix{
         *+[o][F-]{1} \ar@/^/[r]^3
         &
        \bullet \ar@/^/[l]^1} $
     & \begin{minipage}{0.6in}
             $\xymatrix@C=3mm@R=6mm{
                & \bullet \ar[dl]_{1} \ar@/^0.3pc/@<0.7ex>[dr]^{1}           \\
         *+[o][F-]{1}  \ar[rr]_{1} & & \circ
         \ar@/^-0.3pc/@<0.2ex>[ul]^{2}
         }$
         \end{minipage}
     &\begin{minipage}{0.6in}
             $\xymatrix@C=3mm@R=6mm{
                & \bullet \ar[dl]_{2}            \\
         \circ  \ar[rr]_{2} & & \circ
         \ar[ul]_{2}
         }$
         \end{minipage}
\\
\hline $-1/6$ & $-1/2$ & $-1/4$ & $-1/6$ & $1/2$ & $1/8$
\\
\hline  \begin{minipage}{0.6in}
             $\xymatrix@C=3mm@R=6mm{
                & \bullet \ar[dl]_{1}            \\
         *+[o][F-]{1}  \ar[rr]_{2} & & \circ
         \ar[ul]_{2}
         }$
         \end{minipage}
     & \begin{minipage}{0.6in}
             $\xymatrix@C=3mm@R=6mm{
                & \bullet \ar@/^0.3pc/@<0.7ex>[dr]^{2}           \\
         *+[o][F-]{1} \ar[ur]^{1}  & & \circ
         \ar@/^-0.3pc/@<0.2ex>[ul]^{1}\ar[ll]^{1}
         }$
         \end{minipage}
  &  \begin{minipage}{0.6in}
             $\xymatrix@C=3mm@R=6mm{
                & \bullet \ar@/^-0.3pc/@<0.2ex>[dl]^{1} \ar@/^0.3pc/@<0.2ex>[dr]^{1}          \\
         \circ \ar@/^0.3pc/@<0.7ex>[ur]^{1} \ar@/^-0.3pc/@<0.2ex>[rr]^{1} & &
         \circ
         \ar@/^0.3pc/@<0.7ex>[ll]^{1} \ar@/^-0.3pc/@<0.7ex>[ul]^{1}
         }$
         \end{minipage}
     & \begin{minipage}{0.6in}$\xymatrix@C=3mm@R=6mm{
                & \bullet \ar[dr]^{2}             \\
         *+[o][F-]{1} \ar[ur]^{1}  & &  \circ  \ar[ll]^{2}
         }$
         \end{minipage}
     &  \begin{minipage}{0.6in}
             $\xymatrix@C=3mm@R=6mm{
                & \bullet \ar@/^-0.3pc/@<0.2ex>[dl]^{1} \ar@/^0.3pc/@<0.2ex>[dr]^{1}          \\
          *+[o][F-]{1} \ar@/^0.3pc/@<0.7ex>[ur]^{1} & &
          *+[o][F-]{1}
          \ar@/^-0.3pc/@<0.7ex>[ul]^{1}
         }$
         \end{minipage}
     &  \begin{minipage}{0.6in}
             $\xymatrix@C=2mm@R=6mm{
                & \bullet\,1 \ar[dl]_{1}            \\
         *+[o][F-]{1} \ar[rr]_{1} & & *+[o][F-]{1}
         \ar[ul]_{1}
         }$
         \end{minipage}
\\
\hline $1/4$ & $1/2$ & $1/2$ & $1/4$ & $1/2$ & $1$
\\
\hline  \begin{minipage}{0.6in}
             $\xymatrix@C=3mm@R=6mm{
                & \bullet \ar[dl]_{2}            \\
         *+[o][F-]{1}  \ar[rr]_{1} & & *+[o][F-]{1}
         \ar[ul]_{1}
         }$
         \end{minipage}
     & \begin{minipage}{0.6in}
             $\xymatrix@C=3mm@R=6mm{
                & \bullet \ar@/^0.3pc/@<0.7ex>[dr]^{1} \ar[dl]_{1}           \\
         *+[o][F-]{1} \ar[rr]_{1}  & & *+[o][F-]{1}
         \ar@/^-0.3pc/@<0.2ex>[ul]^{1}
         }$
         \end{minipage}
  &  \begin{minipage}{0.6in}
             $\xymatrix@C=3mm@R=6mm{
                & \bullet \ar@/^-0.3pc/@<0.2ex>[dl]^{1}           \\
         *+[o][F-]{1} \ar@/^0.3pc/@<0.7ex>[ur]^{1} \ar[rr]_{1} & &
         *+[o][F-]{1}
          \ar[ul]_{1}
         }$
         \end{minipage}
     & \begin{minipage}{0.6in}
             $\xymatrix@C=3mm@R=6mm{
                & \bullet \ar[dl]_{1}            \\
         *+[o][F-]{1} \ar[rr]_{2} & & *+[o][F-]{1}
         \ar[ul]_{1}
         }$
         \end{minipage}
     &  \begin{minipage}{0.6in}
             $\xymatrix@C=3mm@R=6mm{
                & \bullet \ar[dl]_{1}            \\
         *+[o][F-]{1} \ar[rr]_{1} & & *+[o][F-]{1}
         \ar[ul]_{2}
         }$
         \end{minipage}
     & \begin{minipage}{0.7in}\xymatrix{
  \bullet  \ar[d]_{1}
                &*+[o][F-]{1}\ar[l]_{1}  \\
*+[o][F-]{1}   \ar[r]_1
                & *+[o][F-]{1}    \ar[u]_1       }
                \end{minipage}
\\
\hline $1/2$ & $1$ & $1$ & $1/2$ & $1/2$ & $-1$
\\
\hline
\end{tabular}
\end{table}

\begin{remark}\label{rm2}
Gammelgaard's universal formula \cite{Gam} was expressed as a summation over acyclic graphs with two external vertices (one sink and one source).
All internal vertices are weighted. We refer the reader to \cite{Sch4} for a brief summary of Gammelgaard's work.
In fact, if we identify the two external vertices,
then we get a strongly connected one-pointed graph $\Gamma$ such that $\Gamma_-$ is acyclic.
In the case of Berezin-Toeplitz quantization, by Karabegov-Schlichenmaier's identification theorem, a vertex of weight $-1$
is just a single vertex without loops and a vertex of weight $0$ comes from the Ricci curvature, which may be regarded as
a vertex with exactly one loop. There are no vertices of weight greater than $0$.
If we repeatedly take derivatives on a vertex of weight $0$ by Leibniz rule, we will get a cycle (see \eqref{eqginv} and Figure \ref{figedge}).
Finally we get a graph in $\mathcal G_{BT}$. It is not difficult to prove that Gammelgaard's formula for Berezin-Toeplitz quantization is equivalent
to \eqref{eqBT}
in Theorem \ref{main} by using Lemma \ref{graph2} and Remark \ref{rm3}. For general quantizations,
we may roughly say that
the cycles in the graphs were absorbed into vertices with weight $\geq 0$ in Gammelgaard's formula.
\end{remark}

By Lemma \ref{graph2} and \eqref{eqb3}, in order to prove \eqref{eqb5}, we need only prove the following purely graph theoretic theorem.
\begin{theorem}
Let $G\in\mathcal G_1^{scon}$ be a nontrivial strongly connected one-pointed graph. Then
\begin{equation}\label{eqg1}
\sum_{H\in \mathcal B(G)}(-1)^{|E(H)|}\det\big(A\big((G/H)_-\big)-I\big)=0,
\end{equation}
where $\mathcal B(G)$ consists of all strongly connected one-pointed subgraphs $H$ of $G$ with $H\in\mathcal G_{BT}^{scon}$.
\end{theorem}
\begin{proof}
First we consider a specific graph $G$ of the form as depicted in Figure \ref{fig}.
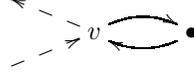
\begin{figure}[h]
\begin{tabular}{c}$\xymatrix@C=9mm@R=2mm
{&   & \\
& v \ar@/^/[r]\ar@{-->}[ul]  & \bullet \ar@/^/[l]\\
\ar@{-->}[ur] & & }$ \end{tabular}
\caption{A one-pointed graph}
 \label{fig}
\end{figure}
Namely there exists an ordinary vertex
$v$ of $G$ such that there is no other edge connected to $\bullet$ besides an edge from $v$ to $\bullet$
and an edge from $\bullet$ to $v$.
Then the graph in $\mathcal B(G)$ is either the trivial graph $\bullet$ or the graph with a cycle of
length $\ell\geq0$ attached at $v$. By Theorem \ref{ct}, the left-hand side of \eqref{eqg1} is equal to
\begin{align*}
\det(A(G_-)-I)+\sum_{L\in\mathscr L(G_-)}(-1)^{\ell(C_v)}\prod_{C\in L\atop v\notin C}(-1)^{\ell(C)+1}=0,
\end{align*}
where $C_v$ is the cycle in $L$ containing $v$.

For general graph $G\in\mathcal G_1^{scon}$ and a fixed subgraph $H\in\mathcal B(G)$, the strongly connected components
of $H$ and a
spanning generalized linear subgraph of $(G/H)_-$ together constitute a spanning generalized linear subgraph of $G_-$.
On the other hand, for a fixed spanning generalized linear subgraph $L\in\mathscr L(G_-)$, we may compute the contributions
of the summation in \eqref{eqg1} to $L$.
In fact, we may contract each cycle in $L$ to a vertex, then it is not difficult to see from the following Lemma \ref{graph3}
that the contributions add up to zero for each $L\in\mathscr L(G_-)$. So we proved \eqref{eqg1}.
\end{proof}

\begin{lemma}\label{graph3}
Let $G\in\mathcal G_1^{scon}$ be a nontrivial strongly connected one-pointed graph. Then
\begin{equation}\label{eqg2}
\sum_{H=(V\cup\{f\},E)\in \mathcal S(G)}(-1)^{w(H)}=0,
\end{equation}
where $\mathcal S(G)$ consists of all strongly connected one-pointed subgraphs $H$ of $G$ with $H\in\mathcal G_{S}^{scon}$
and $w(H)=|E|-|V|$ is the weight of $H$.
\end{lemma}
\begin{proof}
First we note that the lemma still holds for any one-pointed graph $G\in \overline{\mathcal G}_1$ such that
the strongly connected component containing the distinguished vertex $\bullet$ has at least one edge.

Next we prove that if the distinguished vertex of $G$ has $k>0$ loops, denoted by $\{e_1,\dots,e_k\}$, then
\eqref{eqg2} holds. Denote by $\mathcal S'(G)$ the graphs $H$ in $\mathcal S(G)$ such that $H$ has no loops
at the distinguished vertex. Then each graph in $\mathcal S(G)$ is obtained by attaching some loops in
$\{e_1,\dots,e_k\}$ to a graph in $\mathcal S'(G)$. So the left-hand side of \eqref{eqg2} is equal to
$$\sum_{H\in \mathcal S'(G)}\sum_{i=0}^k\binom{k}{i}(-1)^{w(H)+i},$$
which is zero if $k>0$.

In order to prove \eqref{eqg2}, we may assume that $G$ has no loops or multi-edges. In fact, if there are $k>1$
edges from $v_1$ to another vertex $v_2$, denote by $G'$ the graph obtained from $G$ by merging the $k$ edges
a single edge $e$, we have
\begin{align*}
\sum_{H\in \mathcal S(G)}(-1)^{w(H)}&=\sum_{H\in \mathcal S(G')\atop e\notin H}(-1)^{w(H)}
+\sum_{H\in \mathcal S(G')\atop e\in H}\sum_{i=1}^k\binom{k}{i}(-1)^{w(H)+i-1}\\
&=\sum_{H\in \mathcal S(G')\atop e\notin H}(-1)^{w(H)}
+\sum_{H\in \mathcal S(G')\atop e\in H}(-1)^{w(H)}\\
&=\sum_{H\in \mathcal S(G')}(-1)^{w(H)}.
\end{align*}

We may also assume that $G_-$ is acyclic. This is because $G$ can be written as a finite union of maximal strongly connected subgraphs $G'$
with $G'_-$ acyclic
and each $H\in\mathcal B(G)$ must lies in at least one of these $G'$, so we can apply the inclusion-exclusion principle and use
induction on the number of edges of $G$.

If $G_-$ is acyclic, then there exists an ordinary vertex $v$ which is a source of $G_-$, i.e. all edges entering $v$
must come from the distinguished vertex $\bullet$. Since $G$ is strongly connected, so there exists an edge $e$ from $\bullet$
to $v$. If we contract $e$ and absorbs $v$ into $\bullet$ in $G$, we get a new strongly connected graph $G'$ with less number
of vertices. It is not difficult
to see that $H\in \mathcal S(G)$ are in one-to-one correspondence with $H'\in \mathcal S(G')$ given by
$$H'=\begin{cases} H
&\mbox{if } e\notin H;
\\ H/\{e\} & \mbox{if } e\in H,
\end{cases}$$ 
where $H/\{e\}$ is obtained by contracting $e$ in $H$. Moreover $w(H)=w(H')$.
Namely we have
$$\sum_{H\in \mathcal S(G)}(-1)^{w(H)}=\sum_{H'\in \mathcal S(G')}(-1)^{w(H')}.$$ 
So the lemma follows by induction.
\end{proof}

Now we verify the associativity of Berezin-Toeplitz star product directly from \eqref{eqBT}.
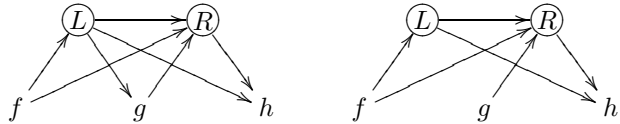
\begin{figure}[h]
$\xymatrix@C=4mm@R=7mm
{& *+[o][F-]{L}\ar[rr]\ar[drrr]\ar[dr] & &*+[o][F-]{R} \ar[dr] &\\
 f\ar[ur]\ar[urrr] & &g\ar[ur] & &h} $
 \qquad
$\xymatrix@C=4mm@R=7mm
{& *+[o][F-]{L}\ar[rr]\ar[drrr] & &*+[o][F-]{R} \ar[dr] &\\
 f\ar[ur]\ar[urrr] & &g\ar[ur] & &h}$
 \caption{Two graphs with 3 distinguished vertices} \label{figthreept}
\end{figure}

\begin{proposition} For any three functions $f,g,h$ on a K\"ahler manifold, we have
\begin{equation}\label{eqb6}
(f\star_{BT} g)\star_{BT} h =f\star_{BT}(g\star_{BT} h).
\end{equation}
\end{proposition}
\begin{proof}
By \eqref{eqBT},
\begin{equation}\label{eqb7}
f\star_{BT} g=\sum_{\Gamma\in\mathcal G_1^{scon}}
\nu^{w(h)}\frac{(-1)^{|E|}}{{\rm Aut}(\Gamma)} \Gamma(f, g).
\end{equation}

We need to prove that for any strongly connected three-pointed graph $G=(V\cup\{f,g,h\},E)$, the
coefficients of $G$ at both sides of \eqref{eqb6} are equal.
All graphs $G$ that may appear in \eqref{eqb6} has two typical forms as depicted in in Figure \ref{figthreept} (each arrow may represent multi-edges).

We denote by $\dot G$ the one-pointed graph obtained by
merging the 3 distinguished vertices of $G$ into one vertex (see Definition \ref{dfdot}).

If $G$ is the first graph, then the contributions of the two sides of \eqref{eqb6} come from $\dot R(\dot L(f,g),h)$ and $\dot L(f,\dot R(g,h))$ respectively,
the choices of subgraphs $L$ and $R$ in $G$ are unique by their strongly connectedness. So they are equal by Lemma \ref{graph2} and \eqref{eqb7}.

If $G$ is the second graph, then the contributions of the two sides of \eqref{eqb6} come from $\dot G(\dot fg,h)$ and $\dot L(f,\dot R(g,h))$ respectively,
which are again equal by Lemma \ref{graph2} and \eqref{eqb7}.
\end{proof}

\vskip 30pt
\section{Karabegov form of a star product} \label{Kara}

Let $(M,\omega_{-1})$ be a pseudo-K\"ahler
manifold, i.e. we only assume $\omega_{-1}$ to be a nondegenerate closed
$(1,1)$-form. A formal deformation of the form $(1/\nu)\omega_{-1}$ is a formal
$(1,1)$-form,
\begin{equation}\label{eqomega}
  \hat\omega = \frac{1}{\nu} \omega_{-1} + \omega_0 + \nu\omega_1 + \nu^2\omega_2
  + \cdots,
\end{equation}
where each $\omega_k$ is a closed $(1,1)$-form. Karabegov \cite{Kar} has
shown that deformation quantizations with separation of variables on
the pseudo-K\"ahler
manifold $(M,\omega_{-1})$ are bijectively parametrized by such formal
deformations.

Given a star product $\star$ of anti-Wick type, its Karabegov form is computed as following:
Let $z^1, \ldots, z^n$ be local holomorphic coordinates on
an open subset $U$ of $M$. Karabegov proved that there exists a set of formal functions on $U$, denoted by
$u^1, \ldots, u^n$,
\begin{align*}
  u^k = \frac{1}{\nu} u^k_{-1} + u^k_0 + \nu u^k_1  + \nu^2 u^k_2
 + \cdots,
\end{align*}
satisfying
$u^k \star z^l - z^l \star u^k = \delta^{kl}$.
Then the classifying Karabegov form of $\star$, which is a global form on
$M$, is given by $\hat\omega \vert_U = - \sqrt{-1}\,\bar\partial (\sum_k u^k dz^k)$
on the coordinate neighborhood $U$.

In a remarkable paper \cite{KS}, Karabegov and Schlichenmaier identified the corresponding Karabegov forms for Berezin and Berezin-Toeplitz star
products on compact K\"ahler manifolds.
Their original proof uses microlocal analysis to obtain off-diagonal expansion of
Bergman kernel, then they derived integral representations of Berezin transform and a newly introduced twisted product in terms of off-diagonal Bergman kernels.
The integral representations were used to prove the existence of their asymptotic expansions. Finally the Berezin-Toeplitz star product was identified by
applying Schlichenmaier's
Theorem \ref{sch}, and the Berezin star product was identified using Zelditch's theorem \cite{Zel}.

We hope our approach through the graph theoretic formulae of Berezin transform
and Bergman kernel could provide further insights to Karabegov-Schlichenmaier's identification theorem.
\begin{theorem}[\cite{KS}]\label{thmKS}
{\rm (i)}
The Karabegov form of Berezin-Toeplitz star product is
$$-\frac1\nu\omega_{-1}+Ric,$$
where $Ric=\sqrt{-1}\partial\bar\partial\log\det g$ is the Ricci curvature.

{\rm (ii)}
The Karabegov form of Berezin star product is
$$\frac1\nu\omega_{-1}+\sqrt{-1}\,\partial\bar\partial\log B(x),$$
where $B(x)$ is the Bergman kernel given in \eqref{eqb2}.
\end{theorem}

\begin{proof}
Let $\omega_{-1}=\sqrt{-1}\partial\bar\partial\Phi$, i.e. $\Phi$ is the potential function for the pseudo-K\"ahler metric $\omega_{-1}$.

First we prove Theorem \ref{thmKS} (i). Since $\hat\omega \vert_U = - \sqrt{-1}\,\bar\partial (\sum_k u^k dz^k)$, it is sufficient to prove that
\begin{equation}
u^k=-\frac1\nu\frac{\partial \Phi}{\partial z^k}
+\frac{\partial\log\det g}{\partial z^k}=-\frac1\nu\frac{\partial \Phi}{\partial z^k}
+g_{i\bar i k}
\end{equation}
satisfy the equation
\begin{multline}
\delta^{kl}=u^k\star_{BT}^{op} z^l-z^l u^k=\sum_{\Gamma\in\mathcal G_{BT} \atop w(\Gamma)>0}\nu^{w(\Gamma)}\frac{(-1)^{|E(\Gamma)|}}{|{\rm Aut}(\Gamma)|}\Gamma^{op}(u^k,z^l)
\\ =\sum_{d=0}^\infty\nu^d \sum_{\Gamma\in\mathcal G_{BT}(d+1)\atop \deg^+(\bullet)=\deg^-(\bullet)=1}\frac{(-1)^{|E(\Gamma)|}}{|{\rm Aut}(\Gamma)|}\Gamma^{op}\left(-\frac{\partial \Phi}{\partial z^k},z^l\right)
\\+\sum_{d=1}^\infty\nu^d \sum_{\Gamma\in\mathcal G_{BT}(d)\atop \deg^-(\bullet)=1}\frac{(-1)^{|E(\Gamma)|}}{|{\rm Aut}(\Gamma)|}\Gamma^{op}\left(g_{i\bar i k},z^l\right)
\end{multline}
for $1\leq k,l\leq n$, where $\star_{BT}^{op}$
is the opposite Berezin-Toeplitz star product of anti-Wick type, i.e. the holomorphic and anti-holomorphic variables
are swapped.

Let us compute the first few coefficients of $u^k\star_{BT} z^l-z^l u^k$ by using \eqref{eqb8}.
It is easy to see that the constant term (the coefficient of $\nu^0$) is equal to
$$-\bigg[\xymatrix{\bullet \ar@(ur,dr)[]^{1}}\bigg]^{op}\left(-\frac{\partial \Phi}{\partial z^k},z^l\right)
=\frac{\partial^2\Phi}{\partial z^k\partial\bar z^l}=\delta^{kl}.$$
The coefficient of $\nu$ (keeping only stable graphs) is equal to
\begin{align*}
-\bigg[\xymatrix{
        *+[o][F-]{1} \ar@/^/[r]^1
         &
        \bullet \ar@/^/[l]^1} \bigg]^{op}\left(-\frac{\partial \Phi}{\partial z^k},z^l\right)
        -&\bigg[\xymatrix{\bullet \ar@(ur,dr)[]^{1}}\bigg]^{op}\left(g_{i\bar i k},z^l\right)\\
&=\bigg[\begin{minipage}{0.5in}$\xymatrix@C=2mm@R=5mm{
         \ar[dr]_{\bar l}
         & & \\
       & *+[o][F-]{1} \ar[ur]_k &}$
         \end{minipage} \bigg]-\bigg[\begin{minipage}{0.5in}$\xymatrix@C=2mm@R=5mm{
         \ar[dr]_{\bar l}
         & & \\
       & *+[o][F-]{1} \ar[ur]_k &}$
         \end{minipage} \bigg]=0.
\end{align*}
The coefficient of $\nu^2$ (keeping only stable graphs) is equal to
\begin{align*}
-\Bigg[\begin{minipage}{0.62in}$\xymatrix@C=2mm@R=5mm{
                & \bullet \ar[dr]^{1}             \\
        *+[o][F-]{1} \ar[ur]^{1}  & &  *+[o][F-]{1} \ar[ll]^{1}
         }$
         \end{minipage} \Bigg]^{op}\left(-\frac{\partial \Phi}{\partial z^k},z^l\right)
        -&\bigg[\xymatrix{
        *+[o][F-]{1} \ar@/^/[r]^1
         &
        \bullet \ar@/^/[l]^1}\bigg]^{op}\left(g_{i\bar i k},z^l\right)\\
&=\Bigg[\begin{minipage}{0.62in}$\xymatrix@C=6mm@R=6mm{
         \ar[d]_{\bar l} &  \\
       *+[o][F-]{1} \ar[r]& *+[o][F-]{1} \ar[u]_k}$
         \end{minipage} \Bigg]-\Bigg[\begin{minipage}{0.62in}$\xymatrix@C=6mm@R=6mm{
         \ar[d]_{\bar l} &  \\
       *+[o][F-]{1} \ar[r]& *+[o][F-]{1} \ar[u]_k}$
         \end{minipage} \Bigg]=0.
\end{align*}

In general, the contribution of $u^k\star_{BT}^{op} z^l-z^l u^k$ to the graph $H$ in Figure \ref{figB}
comes from $\dot H^{op}(-\frac{\partial\Phi}{\partial z^k},z^l)$, where $\dot H\in\mathcal G_{BT}$ is obtained
from $H$ by gluing the head of $k$ and the tail of $\bar l$, and $\Gamma^{op}(g_{i\bar i k},z^l)$,
where $\Gamma$ is a one-pointed graph from Berezin-Toeplitz star product.
So by \eqref{eqb3} and Lemma \ref{graph2}, the coefficient of $H$ is equal to
$$-1\cdot(-1)^{|E(\dot H)|}+(-1)^{|E(\Gamma)|}=0,$$
since $|E(\dot H)|=|E(\Gamma)|+2$. Note that all $H$ that may appear in $u^k\star_{BT}^{op} z^l-z^l u^k$ is of
the form in Figure \ref{figBT}. So we proved $u^k\star_{BT}^{op} z^l-z^l u^k=\delta^{kl}$.
\begin{figure}[h]
$\xymatrix@C=2mm@R=7mm{ &\ar[dl]_{\bar l} &&\\
 *+++[o][F-]{\Gamma} \ar@/^/@{->} @<1pt> [rrr] \ar
@{-->}[rrr] \ar@/_/@{->} @<-1pt> [rrr] &&& *++[o][F-]{1} \ar[ul]_k}$
\caption{A graph $H$ for Berezin-Toeplitz quantization}
\label{figBT}
\end{figure}
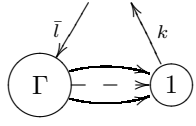

Now we prove Theorem \ref{thmKS} (ii). By \eqref{eqb2}, we need to check that
\begin{equation}
u^k=\frac1\nu\frac{\partial \Phi}{\partial z^k}
+\sum_{G=(V,E)\in\mathcal G_0^{ss}}\nu^{|E|-|V|}\; \frac{-\det(A(G)-I)}{|{\rm Aut}(G)|}\frac{\partial G}{\partial z^k}
\end{equation}
satisfy $u^k\star_B z^l-z^l u^k=\delta^{kl}$ for $1\leq k,l\leq n$.

By \eqref{eqb9}, the constant term (the coefficient of $\nu^0$) in $u^k\star_B z^l-z^l u^k$ is equal to
$$\bigg[\xymatrix{\bullet \ar@(ur,dr)[]^{1}}\bigg]^{op}\left(\frac{\partial \Phi}{\partial z^k},z^l\right)
=\frac{\partial^2\Phi}{\partial z^k\partial\bar z^l}=\delta^{kl}.$$
The coefficient of $\nu$ is vacuously zero.
The coefficient of $\nu^2$ (keeping only stable graphs) is equal to
\begin{multline*}
\left(-\frac12\Bigg[\begin{minipage}{0.6in}$\xymatrix@C=1.5mm@R=5mm{
                & \bullet \ar[dr]^{1}             \\
         \circ \ar[ur]^{1} \ar@/^0.3pc/[rr]^{1} & &  \circ  \ar@/^0.3pc/[ll]^{2}
         }$
         \end{minipage}\Bigg]
-\Bigg[\begin{minipage}{0.65in}
             $\xymatrix@C=1.5mm@R=5mm{
                & \bullet \ar@/^-0.3pc/@<0.2ex>[dl]^{1}             \\
         \circ \ar@/^0.3pc/@<0.7ex>[ur]^{1} \ar@/^-0.3pc/@<0.2ex>[rr]^{1} & & *+[o][F-]{1}
         \ar@/^0.3pc/@<0.7ex>[ll]^{1}
         }$
         \end{minipage}\Bigg]
+\frac12\Bigg[\xymatrix@C=5mm{
        *+[o][F-]{2} \ar@/^/[r]^1
         &
        \bullet \ar@/^/[l]^1} \Bigg]
\right)^{op}\left(\frac{\partial \Phi}{\partial z^k},z^l\right)
\\+\bigg[\xymatrix{\bullet \ar@(ur,dr)[]^{1}}\bigg]^{op}\left(-\frac12\Big[\xymatrix@C=5mm{
        *+[o][F-]{2} \ar[r]^k
         &}\Big]+\bigg[\xymatrix@C=5mm{
        *+[o][F-]{1} \ar@/^/[r]^1
         &
        \circ \ar@/^/[l]^1 \ar[r]^k &}\bigg]+\frac12\bigg[\xymatrix@C=5mm{
        \circ \ar@/^/[r]^2
         &
        \circ \ar@/^/[l]^1 \ar[r]^k &}\bigg],z^l\right),
\end{multline*}
which is easily seen to be zero.

In general, the contribution of $u^k\star_B z^l-z^l u^k$ to the graph $H$ in Figure \ref{figB}
comes from $\dot H^{op}(\frac{\partial\Phi}{\partial z^k},z^l)$, where $\dot H$ is obtained
from $H$ by gluing the head of $k$ and the tail of $\bar l$, and $\Gamma^{op}(\frac{\partial G}{\partial z^k},z^l)$,
where $G$ is a graph from Bergman kernel and $\Gamma$ is a one-pointed graph from Berezin star product.
We may assume that $\frac{\partial}{\partial z^k}$ may act
only on vertices of $G$, since we can extend \eqref{eqb2} to be a summation
over strongly connected graphs (not necessarily semistable). See Remark \ref{rm3}.
So by \eqref{eqb3} and Lemma \ref{graph2}, the coefficient of $H$ is equal to
$$\det(A(\dot H_-)-I)+\det(A(\Gamma_-)-I)\Big(-\det(A(G)-I)\Big)=0.$$
So we proved $u^k\star_B z^l-z^l u^k=\delta^{kl}$.
\begin{figure}[h]
$\xymatrix@C=2mm@R=7mm{ &\ar[dl]_{\bar l} &&\\
 *+++[o][F-]{\Gamma} \ar@/^/@{->} @<1pt> [rrr] \ar
@{-->}[rrr] \ar@/_/@{->} @<-1pt> [rrr] &&& *++[o][F-]{G} \ar[ul]_k}$
\caption{A graph $H$ for Berezin star product}
\label{figB}
\end{figure}
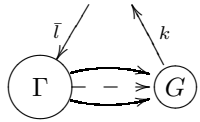
\end{proof}

\begin{remark}\label{rm}
The Karabegov form of an anti-Wick star product has leading term $\frac1\nu\omega_{-1}$ or $-\frac1\nu\omega_{-1}$ depending on
whether its coefficient of $\bigg[\xymatrix{\bullet \ar@(ur,dr)[]^{k}}\bigg]$ is $\frac{1}{k!}$ or $\frac{(-1)^k}{k!}$, in order for \eqref{eqberdef} to hold. In other words,
given an anti-Wick star product on a pseudo-K\"ahler $(M,\omega_{-1})$ with Karabegov form \eqref{eqomega}, we can multiply
$(-1)^k$ on its weight $k$ coefficients to get another anti-Wick star product on $(M,-\omega_{-1})$ with Karabegov form
\begin{equation}\label{eqomega2}
  \hat\omega = -\frac{1}{\nu} \omega_{-1} + \omega_0 - \nu\omega_1 + \nu^2\omega_2 -\nu^3\omega_3
  + \cdots.
\end{equation}
\end{remark}

\vskip 30pt
\section{Karabegov-Bordemann-Waldmann quantization} \label{KBW}

Karabegov-Bordemann-Waldmann (KBW) quantization, denoted by $\star_S$, is also called the standard deformation quantization on K\"ahler manifolds,
i.e. it corresponds to the trivial Karabegov form $\frac1{\nu}\omega_{-1}$ in Karabegov's classification \cite{Kar}. It was
also independently constructed by Bordemann and Waldmann \cite{BW} by modifying Fedosov's geometric construction of
star products on symplectic manifolds \cite{Fed}. The identification
of Bordemann and Waldmann's star product was due to Karabegov \cite{Kar2}. Neumaier \cite{Neu} showed that each star product
 of separation of variables type
can be obtained following Bordemann-Waldmann's construction.

The following theorem is due to Gammelgaard \cite{Gam}. It can also be deduced from recent works of Karabegov \cite{Kar4, Kar5}.
\begin{theorem}[\cite{Gam}]
On a K\"ahler manifold, the KBW star product $\star_S$ is given by
\begin{equation}
f_1\star_{S} f_2(x)=\sum_{\Gamma=(V\cup\{f\},E)\in\mathcal G_S}\nu^{|E|-|V|}\frac{(-1)^{|V|}}{|{\rm Aut}(\Gamma)|}
\Gamma(f_1,f_2)\Big|_x,\label{eqS5}
\end{equation}
where $\mathcal G_S^{scon}$ is the set of strongly connected one-pointed graphs $\Gamma$ such that each
strongly connected component of $\Gamma_-$ is a single vertex without loops, i.e. $\Gamma_-$ is acyclic.
\end{theorem}
In other words,
the formal Berezin transform of KBW quantization is
\begin{equation} \label{eqS2}
I_S(f)=\sum_{\Gamma=(V\cup\{f\},E)\in\mathcal G_S^{scon}}\nu^{|E|-|V|}\frac{(-1)^{|V|}}{|{\rm Aut}(\Gamma)|}
\Gamma,
\end{equation}

We write out the terms in \eqref{eqS2} up to $\nu^4$ (keeping only stable graphs),
\begin{multline*}
I_S(f)=f+\nu\bigg[\xymatrix{\bullet \ar@(ur,dr)[]^{1}}\bigg]+\frac{\nu^2}{2}\bigg[\xymatrix{\bullet \ar@(ur,dr)[]^{2}}\bigg]
+\nu^3\left(\frac16\bigg[\xymatrix{\bullet \ar@(ur,dr)[]^{3}}\bigg]-\frac14\bigg[\xymatrix{ \circ \ar@/^/[r]^2 & \bullet \ar@/^/[l]^2}\bigg]\right)\\
+\nu^4\left(\frac1{24}\bigg[\xymatrix{\bullet \ar@(ur,dr)[]^{4}}\bigg]-\frac14\bigg[\xymatrix{ \circ \ar@/^/[r]^2 & \bullet\, 1 \ar@/^/[l]^2}\bigg]
-\frac{1}{12}\bigg[\xymatrix{ \circ \ar@/^/[r]^2 & \bullet \ar@/^/[l]^3}\bigg]\right.\\\left.-\frac{1}{12}\bigg[\xymatrix{ \circ \ar@/^/[r]^3 & \bullet \ar@/^/[l]^2}\bigg]
+\frac18\Bigg[\begin{minipage}{0.6in}$\xymatrix@C=2mm@R=5mm{
                & \bullet \ar[dr]^{2}             \\
        \circ \ar[ur]^{2}  & &  \circ \ar[ll]^{2}
         }$
         \end{minipage}\Bigg]\right)+O(\nu^5),
\end{multline*}
which agrees with the computations by Karabegov \cite{Kar5}.

The coefficient of $\nu^5$ in \eqref{eqS2} has 15 terms corresponding to graphs in Table \ref{tb3}.
 \begin{table}[h] \footnotesize
\centering \caption{The 15 graphs in $\mathcal G_S(5)$} \label{tb3}
\begin{tabular}{|c|c|c|c|c|}

\hline $\xymatrix{\bullet\,5} $
     & $\xymatrix{
         \circ \ar@/^/[r]^2
         &
        \bullet\,2 \ar@/^/[l]^2} $
      & $\xymatrix{
         \circ \ar@/^/[r]^2
         &
        \bullet\,1 \ar@/^/[l]^3} $
     &  $\xymatrix{
         \circ \ar@/^/[r]^2
         &
        \bullet \ar@/^/[l]^4} $
     & $\xymatrix{
         \circ \ar@/^/[r]^3
         &
        \bullet\,1 \ar@/^/[l]^2} $
\\
\hline $1/120$ & $-1/8$ & $-1/12$ & $-1/48$ & $-1/12$
\\
\hline $\xymatrix{
         \circ \ar@/^/[r]^3
         &
        \bullet \ar@/^/[l]^3} $
     & \xymatrix{
         \circ \ar@/^/[r]^4
         &
        \bullet \ar@/^/[l]^2}
     & \begin{minipage}{0.6in}
             $\xymatrix@C=3mm@R=6mm{
                & \bullet \ar@/^-0.3pc/@<0.2ex>[dl]^{2} \ar@/^0.3pc/@<0.2ex>[dr]^{1}          \\
         \circ \ar@/^0.3pc/@<0.7ex>[ur]^{1} \ar[rr]_{1} & &
         \circ
         \ar@/^-0.3pc/@<0.7ex>[ul]^{2}
         }$
         \end{minipage}
     &  \begin{minipage}{0.6in}
             $\xymatrix@C=2mm@R=6mm{
                & \bullet\,1 \ar[dl]_{2}            \\
         \circ  \ar[rr]_{2} & & \circ
         \ar[ul]_{2}
         }$
         \end{minipage}
     &   \begin{minipage}{0.6in}
             $\xymatrix@C=3mm@R=6mm{
                & \bullet \ar[dl]_{3}            \\
         \circ  \ar[rr]_{2} & & \circ
         \ar[ul]_{2}
         }$
         \end{minipage}
\\
\hline $-1/36$ & $-1/48$ & $1/4$ & $1/8$ & $1/24$
\\
\hline
       \begin{minipage}{0.6in}
             $\xymatrix@C=3mm@R=6mm{
                & \bullet \ar[dl]_{2} \ar@/^0.3pc/@<0.7ex>[dr]^{1}           \\
         \circ  \ar[rr]_{2} & & \circ
         \ar@/^-0.3pc/@<0.2ex>[ul]^{2}
         }$
         \end{minipage}
  & \begin{minipage}{0.6in}
             $\xymatrix@C=3mm@R=6mm{
                & \bullet \ar@/^-0.3pc/@<0.2ex>[dl]^{2}            \\
         \circ \ar@/^0.3pc/@<0.7ex>[ur]^{1} \ar[rr]_{2} & &
         \circ
         \ar[ul]_{2}
         }$
         \end{minipage}
     & \begin{minipage}{0.6in}
             $\xymatrix@C=3mm@R=6mm{
                & \bullet \ar[dl]_{2}            \\
         \circ  \ar[rr]_{3} & & \circ
         \ar[ul]_{2}
         }$
         \end{minipage}
     &  \begin{minipage}{0.6in}
             $\xymatrix@C=3mm@R=6mm{
                & \bullet \ar[dl]_{2}            \\
         \circ  \ar[rr]_{2} & & \circ
         \ar[ul]_{3}
         }$
         \end{minipage}
     & \begin{minipage}{0.7in}\xymatrix{
  \bullet  \ar[d]_{2}
                &\circ \ar[l]_{2}  \\
\circ   \ar[r]_2
                & \circ    \ar[u]_2       }
                \end{minipage}
\\
\hline $1/8$ & $1/8$ & $1/24$ & $1/24$ & $-1/16$
\\
\hline
\end{tabular}
\end{table}

The dual Karabegov-Bordemann-Waldmann star product $\star_{DS}$ is defined by
\begin{equation}
f_1\star_{DS} f_2=I^{-1}_S(I_S f_2\star_S I_S f_1),
\end{equation}
for any functions $f_1,f_2$.
\begin{theorem}
The formal Berezin transform of the dual KBW quantization is
\begin{equation}\label{eqS3}
I_S^{-1}(f)=\sum_{\Gamma=(V\cup\{f\},E)\in\mathcal G_1^{scon}}\nu^{|E|-|V|}\frac{(-1)^{|E|}}{|{\rm Aut}(\Gamma)|}
\Gamma.
\end{equation}
\end{theorem}
\begin{proof}
By Lemma \ref{graph3}, for any nontrivial strongly connected one-pointed graph $G\in\mathcal G_1^{scon}$, we have
$$\sum_{H=(V\cup\{f\},E)\in \mathcal S(G)}(-1)^{|E(G/H)|}(-1)^{|V|}
=(-1)^{|E(G)|}\sum_{H\in \mathcal S(G)}(-1)^{w(H)}=0.$$
So \eqref{eqS3} follows from \eqref{eqS2}.
\end{proof}

We write out terms of \eqref{eqS3} up to $\nu^3$ (keeping only stable graphs)
\begin{multline*}
I_S^{-1}(f)=f-\nu\bigg[\xymatrix{\bullet
              \ar@(ur,dr)[]^{1}}\bigg]
+\nu^2\left(\frac12\bigg[\xymatrix{\bullet
           \ar@(ur,dr)[]^{2}}\bigg]-\bigg[\xymatrix{
        *+[o][F-]{1} \ar@/^/[r]^1
         &
        \bullet \ar@/^/[l]^1} \bigg]\right)
\\
\nu^3\left(
-\frac{1}{6}\bigg[\xymatrix{\bullet \ar@(ur,dr)[]^{3}}\bigg]
+\bigg[\xymatrix{
        *+[o][F-]{1} \ar@/^/[r]^1
         &
        \bullet\, 1 \ar@/^/[l]^1} \bigg]
+\frac14\bigg[\xymatrix{ \circ \ar@/^/[r]^2 & \bullet \ar@/^/[l]^2
         }\bigg]\right.
\\
+\frac12\bigg[\xymatrix{
        *+[o][F-]{1} \ar@/^/[r]^1
         &
        \bullet \ar@/^/[l]^2} \bigg]
+\frac12\bigg[\xymatrix{
        *+[o][F-]{1} \ar@/^/[r]^2
         &
        \bullet \ar@/^/[l]^1} \bigg]
+\frac12\bigg[\xymatrix{
        *+[o][F-]{2} \ar@/^/[r]^1
         &
        \bullet \ar@/^/[l]^1} \bigg]
\\\left.-\frac12\Bigg[\begin{minipage}{0.6in}$\xymatrix@C=2mm@R=5mm{
                & \bullet \ar[dr]^{1}             \\
         \circ \ar[ur]^{1} \ar@/^0.3pc/[rr]^{1} & &  \circ  \ar@/^0.3pc/[ll]^{2}
         }$
         \end{minipage}\Bigg]
-\Bigg[\begin{minipage}{0.65in}
             $\xymatrix@C=2mm@R=5mm{
                & \bullet \ar@/^-0.3pc/@<0.2ex>[dl]^{1}             \\
         \circ \ar@/^0.3pc/@<0.7ex>[ur]^{1} \ar@/^-0.3pc/@<0.2ex>[rr]^{1} & & *+[o][F-]{1}
         \ar@/^0.3pc/@<0.7ex>[ll]^{1}
         }$
         \end{minipage}\Bigg]
-\Bigg[\begin{minipage}{0.6in}$\xymatrix@C=2mm@R=5mm{
                & \bullet \ar[dr]^{1}             \\
        *+[o][F-]{1} \ar[ur]^{1}  & &  *+[o][F-]{1} \ar[ll]^{1}
         }$
         \end{minipage}\Bigg]\right)
        +O(\nu^4).
\end{multline*}

\begin{corollary}\label{dS} The Karabegov form of $\star_{DS}$ is
\begin{equation}
-\frac{1}{\nu}\omega_{-1}+Ric+\sqrt{-1}\partial\bar\partial\sum_{G\in\mathcal G_0^{ss}}\nu^{w(G)}\frac{(-1)^{|E(G)|+1}}{|{\rm Aut}(G)|}
G,
\end{equation}
where $G$ runs over all strongly connected semistable graphs.
\end{corollary}
\begin{proof}
By \eqref{eqS3}, it is sufficient to prove that
\begin{align}
u_k&=-\frac1\nu\frac{\partial \Phi}{\partial z^k}
+\frac{\partial\log\det g}{\partial z^k}+\sum_{G\in\mathcal G_0^{ss}}\nu^{w(G)}\frac{(-1)^{|E(G)|+1}}{|{\rm Aut}(G)|}\frac{\partial G}{\partial z^k}\\
&=-\frac1\nu\frac{\partial \Phi}{\partial z^k}
+g_{i\bar i k}+\sum_{G\in\mathcal G_0^{ss}}\nu^{w(G)}\frac{(-1)^{|E(G)|+1}}{|{\rm Aut}(G)|}\frac{\partial G}{\partial z^k}
\end{align}
satisfy the equation
\begin{multline}
\delta^{kl}=u^k\star_{DS} z^l-z^l u^k=\sum_{\Gamma\in\mathcal G_1 \atop w(\Gamma)>0}\nu^{w(\Gamma)}\frac{(-1)^{|E(\Gamma)|}}{|{\rm Aut}(\Gamma)|}\Gamma^{op}(u^k,z^l)
\\ =\sum_{d=0}^\infty\nu^d \sum_{\Gamma\in\mathcal G_1(d+1)\atop \deg^+(\bullet)=\deg^-(\bullet)=1}\frac{(-1)^{|E(\Gamma)|}}{|{\rm Aut}(\Gamma)|}\Gamma^{op}\left(-\frac{\partial \Phi}{\partial z^k},z^l\right)
\\+\sum_{d=1}^\infty\nu^d \sum_{\Gamma\in\mathcal G_1(d)\atop \deg^-(\bullet)=1}
\frac{(-1)^{|E(\Gamma)|}}{|{\rm Aut}(\Gamma)|}\Gamma^{op}\left(g_{i\bar i k},z^l\right)
\\+\sum_{d=2}^\infty\nu^d \sum_{t=1}^{d-1} \sum_{\Gamma\in\mathcal G_1(t)\atop \deg^-(\bullet)=1}\sum_{G\in \mathcal G_0^{ss}(d-t)}
\frac{(-1)^{|E(\Gamma)|}}{|{\rm Aut}(\Gamma)|}\Gamma^{op}\left(\frac{(-1)^{E(G)+1}}{|{\rm Aut}(G)|}\frac{\partial G}{\partial z^k},z^l\right).
\end{multline}

Note that we have
\begin{equation}
\sum_{G\in\mathcal G_0^{ss}}\nu^{w(G)}\frac{(-1)^{|E(G)|+1}}{|{\rm Aut}(G)|}
G=1+\nu\left(-\frac{1}{2}\big[\xymatrix{*+[o][F-]{2}}\big]+\frac12\bigg[\xymatrix{ \circ  \ar@/^/[r]^2 & \circ \ar@/^/[l]^1
         }\bigg]\right)+O(\nu^2).\label{eqS4}
\end{equation}

It is easy to see that the constant term (the coefficient of $\nu^0$) is equal to
$$-\bigg[\xymatrix{\bullet \ar@(ur,dr)[]^{1}}\bigg]^{op}\left(-\frac{\partial \Phi}{\partial z^k},z^l\right)
=\frac{\partial^2\Phi}{\partial z^k\partial\bar z^l}=\delta^{kl}.$$
The coefficient of $\nu$ (keeping only stable graphs) is equal to
\begin{align*}
-\bigg[\xymatrix{
        *+[o][F-]{1} \ar@/^/[r]^1
         &
        \bullet \ar@/^/[l]^1} \bigg]^{op}\left(-\frac{\partial \Phi}{\partial z^k},z^l\right)
        -&\bigg[\xymatrix{\bullet \ar@(ur,dr)[]^{1}}\bigg]^{op}\left(g_{i\bar i k},z^l\right)\\
&=\bigg[\begin{minipage}{0.5in}$\xymatrix@C=2mm@R=5mm{
         \ar[dr]_{\bar l}
         & & \\
       & *+[o][F-]{1} \ar[ur]_k &}$
         \end{minipage} \bigg]-\bigg[\begin{minipage}{0.5in}$\xymatrix@C=2mm@R=5mm{
         \ar[dr]_{\bar l}
         & & \\
       & *+[o][F-]{1} \ar[ur]_k &}$
         \end{minipage} \bigg]=0.
\end{align*}

The coefficient of $\nu^2$ (keeping only stable graphs) is equal to
\begin{multline*}
\left(\frac12\Bigg[\begin{minipage}{0.6in}$\xymatrix@C=1.5mm@R=5mm{
                & \bullet \ar[dr]^{1}             \\
         \circ \ar[ur]^{1} \ar@/^0.3pc/[rr]^{1} & &  \circ  \ar@/^0.3pc/[ll]^{2}
         }$
         \end{minipage}\Bigg]
+\Bigg[\begin{minipage}{0.65in}
             $\xymatrix@C=1.5mm@R=5mm{
                & \bullet \ar@/^-0.3pc/@<0.2ex>[dl]^{1}             \\
         \circ \ar@/^0.3pc/@<0.7ex>[ur]^{1} \ar@/^-0.3pc/@<0.2ex>[rr]^{1} & & *+[o][F-]{1}
         \ar@/^0.3pc/@<0.7ex>[ll]^{1}
         }$
         \end{minipage}\Bigg]
-\frac12\Bigg[\xymatrix@C=5mm{
        *+[o][F-]{2} \ar@/^/[r]^1
         &
        \bullet \ar@/^/[l]^1} \Bigg]
+\Bigg[\begin{minipage}{0.6in}$\xymatrix@C=1.2mm@R=5mm{
                & \bullet \ar[dr]^{1}             \\
        *+[o][F-]{1} \ar[ur]^{1}  & &  *+[o][F-]{1} \ar[ll]^{1}
         }$
         \end{minipage} \Bigg]\right)^{op}\left(\frac{\partial \Phi}{\partial z^k},z^l\right)
\\-\bigg[\xymatrix{
        *+[o][F-]{1} \ar@/^/[r]^1
         &
        \bullet \ar@/^/[l]^1}\bigg]^{op}\left(\Big[\xymatrix@C=5mm{
        *+[o][F-]{1} \ar[r]^k
         &}\Big],z^l\right)
\\+\bigg[\xymatrix{\bullet \ar@(ur,dr)[]^{1}}\bigg]^{op}\left(\frac12\Big[\xymatrix@C=5mm{
        *+[o][F-]{2} \ar[r]^k
         &}\Big]-\bigg[\xymatrix@C=5mm{
        *+[o][F-]{1} \ar@/^/[r]^1
         &
        \circ \ar@/^/[l]^1 \ar[r]^k &}\bigg]-\frac12\bigg[\xymatrix@C=5mm{
        \circ \ar@/^/[r]^2
         &
        \circ \ar@/^/[l]^1 \ar[r]^k &}\bigg],z^l\right),
\end{multline*}
which is easily seen to be zero.

In general, the contribution of $u^k\star_{DS} z^l-z^l u^k$ to the graph $H$ in Figure \ref{figDS}
comes from $\dot H^{op}(\frac{\partial\Phi}{\partial z^k},z^l)$, where $\dot H$ is obtained
from $H$ by gluing the head of $k$ and the tail of $\bar l$, and $\Gamma^{op}(\frac{\partial G}{\partial z^k},z^l)$,
where $G$ is the unique sink of $\dot H_-$ containing the tail of $k$ and $\Gamma$ is a one-pointed graph. We may assume that $\frac{\partial}{\partial z^k}$ may act
only on vertices of $G$, since we can extend \eqref{eqS4} to be a summation
over strongly connected graphs (not necessarily semistable). See Remark \ref{rm3}.
So by \eqref{eqS3} and Lemma \ref{graph2}, the coefficient of $H$ is equal to
$$(-1)\cdot(-1)^{|E(\dot H)|}+(-1)^{|E(\Gamma)|}\cdot(-1)^{|E(G)|+1}=0,$$
since $|E(\dot H)|=|E(\Gamma)|+|E(G)|+1$. So we proved $u^k\star_{DS} z^l-z^l u^k=\delta^{kl}$.
\begin{figure}[h]
$\xymatrix@C=2mm@R=7mm{ &\ar[dl]_{\bar l} &&\\
 *+++[o][F-]{\Gamma} \ar@/^/@{->} @<1pt> [rrr] \ar
@{-->}[rrr] \ar@/_/@{->} @<-1pt> [rrr] &&& *++[o][F-]{G} \ar[ul]_k}$
\caption{A graph $H$ for dual KBW star product}
\label{figDS}
\end{figure}
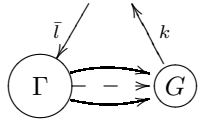
\end{proof}

$$ \ \ \ \ $$

\

\end{document}